\documentclass[a4paper,10pt]{amsart}

\usepackage{
	amsmath,
 amssymb,
 fullpage,
 mathtools, % for \coloneqq
 thmtools,
}

\usepackage[pagebackref, hypertexnames=false]{hyperref}
\usepackage[capitalise]{cleveref}

\theoremstyle{plain}
\newtheorem{theorem}{Theorem}[section]
\newtheorem{lemma}[theorem]{Lemma}
\newtheorem{proposition}[theorem]{Proposition}
\newtheorem{definition}[theorem]{Definition}
\newtheorem{notation}[theorem]{Notation}

\theoremstyle{remark}
\declaretheorem[name=Remark,sibling=theorem,qed={\lower-0.3ex\hbox{$\diamond$}}]{remark}

\DeclareMathOperator{\GSp}{GSp}
\DeclareMathOperator{\Sym}{Sym}
\DeclareMathOperator{\GL}{GL}
\DeclareMathOperator{\Gal}{Gal}
\DeclareMathOperator{\rank}{rank}

\newcommand{\alg}{\mathrm{alg}}
\newcommand{\Iw}{\mathrm{Iw}}

\renewcommand{\AA}{\mathbf{A}}
\newcommand{\CC}{\mathbf{C}}
\newcommand{\QQ}{\mathbf{Q}}
\newcommand{\ZZ}{\mathbf{Z}}

\newcommand{\cO}{\mathcal{O}}
\newcommand{\cT}{\mathcal{T}}

\newcommand{\wH}{\widetilde{H}}

\newcommand{\QQbar}{\overline{\QQ}}
\newcommand{\Qp}{\QQ_p}
\newcommand{\Af}{\AA_{\mathrm{f}}}
\newcommand{\mpi}{\mu_{p^\infty}}

\renewcommand{\le}{\leqslant}
\renewcommand{\ge}{\geqslant}

\numberwithin{equation}{section}

\author{David Loeffler}
\address[Loeffler]{Mathematics Institute\\
 Zeeman Building, University of Warwick\\
 Coventry CV4 7AL, UK.}
\email{d.a.loeffler@warwick.ac.uk}
\urladdr{\href{http://orcid.org/0000-0001-9069-1877}{\upshape\path{0000-0001-9069-1877}}}

\author{Sarah Livia Zerbes}
\address[Zerbes]{Department of Mathematics \\
 University College London\\
 Gower Street, London WC1E 6BT, UK.}
\email{s.zerbes@ucl.ac.uk}
\urladdr{\href{https://orcid.org/0000-0001-8650-9622}{\upshape\path{0000-0001-8650-9622}}}

\title{On the Bloch--Kato conjecture for the symmetric cube}

\thanks{The authors are grateful to acknowledge financial support from the European Research Council (ERC Consolidator Grant ``Euler Systems and the Birch--Swinnerton-Dyer conjecture'') and the Royal Society (University Research Fellowship ``L-functions and Iwasawa theory'').}

\begin{document}
 \begin{abstract}
  We prove one inclusion in the Iwasawa main conjecture, and the Bloch--Kato conjecture in analytic rank 0, for the symmetric cube of a level 1 modular form.
 \end{abstract}
 \maketitle

\section{Introduction}

 In this short note, we study the arithmetic of the $p$-adic symmetric cube Galois representation associated to a level 1 modular cusp form. We prove that the Bloch--Kato conjecture holds in analytic rank $0$ for the critical twists of this representation (corresponding to the critical values of the $L$-function); and we prove one inclusion of the cyclotomic Iwasawa Main Conjecture for this representation, showing that the characteristic ideal of the Selmer group divides the $p$-adic $L$-function. We deduce these results from the results of \cite{LZ20} on the arithmetic of automorphic representations of $\GSp(4)$, using the existence of a symmetric cube lifting from $\GL(2)$ to $\GSp(4)$.

 \begin{remark}\
  \begin{itemize}
   \item The results of this paper were inspired by the striking work of Haining Wang \cite{wang20}, where he proves many cases of the Bloch--Kato conjecture for symmetric cubes of modular forms via arithmetic level-raising on triple products of Shimura curves. Our results are complementary to those of Wang, since we restrict to level 1 (and hence weight $\ge 12$), while he considers weight 2 and general level.

   \item The converse implication of the Bloch--Kato conjecture in this setting -- i.e.~that the vanishing of the central critical $L$-value implies the existence of a non-zero Selmer class -- will be treated in forthcoming work of Samuel Mundy, using a functorial lifting to the exceptional group $G_2$.

   \item The restriction to modular forms of level $1$ is carried over from the results in \emph{op.cit}., which rely on forthcoming work of Barrera, Dimitrov and Williams regarding interpolation of $p$-adic $L$-functions for $\GL(2n)$ in finite-slope families. At present these results are only available for level 1, but we are confident that it should be possible to relax this assumption. In this case, all of the results of this paper will generalize without change to modular forms of higher level.
   \qedhere
  \end{itemize}
 \end{remark}

 In forthcoming work, we will prove the analogous results for the Galois representations attached to quadratic Hilbert modular forms, using the twisted Yoshida lift to $\GSp(4)$.

%%%%%%%%%%%%%%%%%%%%%%%%%%%%%%%%%%%%%%%%%%%%%%%%%%%%%%%%%%%%%%%%%%%%%%%%%%%%%%%%

\section{Large Galois image}

 Let $f$ be a cuspidal, new, non-CM eigenform of weight $k \ge 2$ and level $\Gamma_1(N)$ for some $N \ge 1$, and write $F$ for the coefficient field of $f$. %Assume that $f$ has trivial central character.
 Throughout this section, let $p$ be a prime $>3$. Let $v$ be a prime of $F$ above $p$, and denote by $V_{f,v}$ the 2-dimensional $F_v$-linear Galois representation associated to $f$. We are interested in the representation $\Sym^3 V_{f,v}$.

 \begin{remark}\
  \begin{enumerate}
   \item Our conventions are such that the trace of \emph{geometric} Frobenius at a prime $\ell \nmid pN$ is the Hecke eigenvalue $a_\ell(f)$.
   \item The Hodge--Tate weights of $\left(\Sym^3V_{f,v}\right)|_{G_{\Qp}}$ are $\{3-3k,\, 2-2k,\, 1-k,\, 0\}$.
   \item We have $\left(\Sym^3V_{f,v}\right)^*\cong \left(\Sym^3V_{f,v}\right)(3k-3)$.
   \item  The critical twists of $\Sym^3V_{f,v}$ are $\left(\Sym^3V_{f,v}\right)(j)$ for $k\le j\le 2k-2$. \qedhere
  \end{enumerate}
 \end{remark}

 \begin{proposition}\label{prop:symcubelargeimage}
  If $p\gg 0$, then the $p$-adic representation $W_\Pi^*=\Sym^3V^*_{f,v}$ has large image: there exists a $G_{\QQ}$-invariant lattice $\cT_\Pi$ of $W_\Pi^*$ with the following properties:
  \begin{itemize}
    \item $\cT_\Pi\otimes k_v$ is an irreducible $k_v\left[\Gal(\QQbar/\QQ(\mpi))\right]$-module, where $k_v$ is the residue field of $\cO_{F_v}$, and
    \item there exists  $\tau\in\Gal\left(\QQbar/\QQ(\mpi)\right)$ such that $\rank_{\cO_{F_v}} \cT_\Pi/(\tau-1)\cT_\Pi=1$.
  \end{itemize}
 \end{proposition}

 In order to prove this proposition, we quote the following result due to Ribet \cite{ribet85}:

 \begin{lemma}
  If $p$ is sufficiently large, then there exists a $\Gal(\QQbar/\QQ)$-stable $\cO_{F_v}$-lattice $T$ of $V^*_{f,v}$ such that the image of the homomorphism
  \[ \Gal\left(\QQbar/\QQ(\mpi)\right)\longrightarrow \GL_2(T)\cong \GL_2(\cO_{F_v})\]
  contains $\operatorname{SL}_2(\ZZ_p)$. Here, the last isomorphism is given by a basis of $T$.
 \end{lemma}

 We now prove Proposition \ref{prop:symcubelargeimage}:
 \begin{proof}
  Let $\tau\in  \Gal\left(\QQbar/\QQ(\mpi)\right)$ such that its image in $\GL_2(\cO_{F_v})$ is equal to $\begin{pmatrix}1 & 1 \\ 0 & 1  \end{pmatrix}$. Then the matrix of $\tau$ with respect to the natural basis of $\cT_\Pi=\Sym^3 T$ is given by $\begin{pmatrix} 1 & 3 & 3 & 1\\ & 1 & 2 & 1\\ & & 1 & 1 \\ &&&1\end{pmatrix}$, so since $p>3$, we deduce that $\rank_{\cO_{F_v}} \cT_\Pi/(\tau-1)\cT_\Pi=1$.

  The irreducibility of $\cT_\Pi\otimes k_v$ as a  $k_v\left[\Gal(\QQbar/\QQ(\mpi))\right]$-module is immediate.
 \end{proof}

 \begin{remark}\label{note:Dirchar}
  We can clearly choose $\tau\in  \Gal\left(\QQbar/\QQ^{\mathrm{ab}})\right)$. Hence we have that
  \[ \rank_{\cO_{F_v}} \cT_\Pi(\chi)/(\tau-1)\cT_\Pi(\chi)=1\]
  for every Dirichlet character $\chi$ of prime-to-$p$-conductor.
 \end{remark}

%%%%%%%%%%%%%%%%%%%%%%%%%%%%%%%%%%%%%%%%%%%%%%%%%%%%%%%%%%%%%%%%%%%%%%%%%%%%%%%%

\section{Functoriality}

 \begin{notation}
  Let $\pi$ be the associated automorphic representation of $\GL_2(\Af)$, and write $\Sym^3\pi$ for the image of $\pi$ under the Langlands functoriality transfer from $\GL_2$ to $\GL_4$ attached to $\Sym^3: \GL_2(\CC)\rightarrow \GL_4(\CC)$.
 \end{notation}

 \begin{theorem}
  There exists a globally generic cuspidal automorphic representation $\Pi=\bigotimes_v \Pi_v$ of $\GSp_4(\AA)$ of weight $(2k-1,k+1)$ such that $\Sym^3\pi$ is the functorial lift of $\Pi$ under the embedding $\GSp_4(\CC)\rightarrow \GL_4(\CC)$. In particular, we have
  \begin{enumerate}
   \item $\Pi$ is unramified at the primes not dividing $N$,
   \item $L\left(\Pi,s-\frac{w}{2}\right)=L(\Sym^3\pi,s)$, where $w=3(k+1)$, and
   \item the $p$-adic spin representation $W_\Pi$ associated to $\Pi$ is equal to $\Sym^3V_{f,v}$.
  \end{enumerate}
  Here, we normalise $L(\Pi,s)$ such that the centre of the functional equation is at $s = \tfrac{1}{2}$.
 \end{theorem}
 \begin{proof}
  See \cite{kimshahidi02} and \cite[\S 10]{conti19}.
 \end{proof}

 \begin{remark}
  In more classical terms, $\Pi$ (or more precisely its $L$-packet) corresponds to a holomorphic vector-valued Siegel modular form taking values in the representation $\det^{k+1} \otimes \Sym^{3k-3}$ of $\GL_2(\CC)$.
 \end{remark}

%%%%%%%%%%%%%%%%%%%%%%%%%%%%%%%%%%%%%%%%%%%%%%%%%%%%%%%%%%%%%%%%%%%%%%%%%%%%%%%%

\section{Application to the Bloch--Kato conjecture}

 \begin{theorem}\label{thm:BK}
  Assume that  $f$ is a cuspidal eigenform of weight $k$ and level $1$, and let $V=\left(\Sym^3 V_{f,v}\right)(2k-2)$. Let $p$ be a prime $>3$, and assume that $f$ is ordinary at $p$.  Let $0\leq j\leq k-2$, and let $\rho$ be a finite order character of $\ZZ_p^\times$. If $L\left(\Sym^3\pi\otimes \rho, k+j\right)\neq 0$, then $H^1_{\mathrm{f}}(\QQ,V(-j-\rho))=0$.
 \end{theorem}

 \begin{proof}
  The assumption on $f$ being ordinary at $p$ implies that the representation $\Pi$ is Borel ordinary at $p$. Moreover, Note \ref{note:Dirchar} implies that the big image assumption $\mathrm{Hyp}(\QQ(\mpi),\cT_\Pi(\chi)$ (c.f. \cite[\S. 2]{LZ20}) is satisfied for all Dirichlet characters $\chi$ of prime-to-$p$ conductor.       We are hence in the situation where can apply \cite[Theorem D]{LZ20}.
 \end{proof}

 \begin{remark}
  Observe that if $j \ne \frac{k-2}{2}$ (the central value), then the $L$-value $L\left(\Sym^3\pi\otimes \rho, k+j\right)$ is automatically non-zero, since $L(\Pi, s)$ is non-vanishing for $\operatorname{Re}(s) > 1$.
 \end{remark}

%%%%%%%%%%%%%%%%%%%%%%%%%%%%%%%%%%%%%%%%%%%%%%%%%%%%%%%%%%%%%%%%%%%%%%%%%%%%%%%%

\section{Application to the Iwasawa main conjecture}

 We similarly obtain applications to the cyclotomic Iwasawa Main Conjecture. We preserve the notations and hypotheses of the previous section. Let $\Gamma=\Gal(\QQ(\mpi)/\QQ)$, and write $\Lambda = \cO_{F_v}[[\Gamma]]$ for the Iwasawa algebra.

 \subsection{P-adic $L$-functions for the symmetric cube} We recall the following two known results:

  \begin{theorem}[Algebraicity of $\Sym^3$ $L$-values]
   There exist constants $\Omega_{\Pi}^{+},\Omega_{\Pi}^{-}  \in \CC$, well-defined up to multiplication by $F^\times$, such that for every Dirichlet character $\chi$ and every $0 \le j \le k-2$, the quantity
   \[ L^{\alg}(\Sym^3 f  \otimes \bar\rho, j+k) \coloneqq \frac{G(\rho)^2}{(2\pi i)^j \cdot \Omega^{\varepsilon}_{\Pi}} L(\Sym^3 f \otimes \bar\rho, j+k) \]
   lies in $F(\chi)$, and depends $\Gal(\overline{F} / F)$-equivariantly on $\chi$. Here $\varepsilon = (-1)^j \rho(-1)$.
  \end{theorem}

  \begin{theorem}[Existence of the $\Sym^3$ $p$-adic $L$-function]
   Let $\alpha$ be the unit root of the Hecke polynomial of $f$ at $p$. Then there exists an element $L_{v, \alpha}(\Sym^3 f) \in \Lambda$ with the following interpolation property: for any Dirichlet character $\rho$ of $p$-power conductor, and any $0 \le j \le k-2$, we have
   \[
    L_{v, \alpha}(\Sym^3 f)(j + \rho) = j! (j + k - 1)!\, R_p(\Sym^2 f, \rho, j)\,  L^{\alg}(\Sym^3 f  \otimes \bar\rho, j+k),
   \]
   where
   \[
    R_p(\Sym^3 f, \rho, j)\coloneqq \begin{cases}
     \left(\tfrac{p^{2(j+k-1)}}{\alpha^5 \beta}\right)^m  & \text{if $\rho$ has conductor $p^m > 1$},\\
     \left(1 - \tfrac{p^{j+k-1}}{\alpha^3}\right)\left(1 - \tfrac{p^{j+k-1}}{\alpha^2\beta}\right)\left(1 - \tfrac{\alpha\beta^2}{p^{j+k}}\right)\left(1 - \tfrac{\beta^3}{p^{j+k}}\right) & \text{if $\rho = 1$.}
    \end{cases}
   \]
  \end{theorem}

  These results are due to Dimitrov, Januszewski and Raghuram \cite{DJR18}, as a special case of general theorems applying to any automorphic representation $\Pi$ of $\GL(4)$ admitting a Shalika model. An alternative proof using coherent cohomology of Siegel Shimura varieties is given in \cite[Theorem A]{LPSZ1}.

 \subsection{The Iwasawa main conjecture}

  Recall that $V$ denotes the Galois representation $\left(\Sym^3 V_{f,v}\right)(2k-2)$.

  \begin{definition}
   Let $\widetilde{R\Gamma}_{\Iw}(\QQ(\mpi),V)$ denote the Nekov\'a\v{r} Selmer complex with the unramified local condition at $\ell\neq p$, and the Greenberg local condition at $p$ determined by the ordinarity of $f$.
  \end{definition}

  We shall be interested in the degree 2 cohomology $\wH^2_{\Iw}(\QQ(\mpi),V)$. Note that this can be described in more classical terms: if $T$ is a Galois-stable lattice in $V$, and $(-)^\vee$ denotes Pontryagin duality, then for any $0 \le j \le k-2$, there is a canonical isomorphism
  \[  \wH^2_{\Iw}(\QQ_{\mpi},V) \cong \left(\varinjlim_n H^1_{\mathrm{f}}(\QQ(\mu_{p^n}), T^\vee(1 + j)) \right)^\vee(j) \otimes \Qp. \]

  \begin{theorem}
   If the assumptions of Theorem \ref{thm:BK} are satisfied, then $\wH^2_{\Iw}(\QQ(\mpi),V)$ is $\Lambda[1/p]$-torsion, and its characteristic ideal divides the $p$-adic $L$-function $L_{v, \alpha}(\Sym^3 f)$.
  \end{theorem}

  \begin{proof}
   This follows by applying Theorem C of \cite{LZ20} to the representation $\Pi$.
  \end{proof}

%%%%%%%%%%%%%%%%%%%%%%%%%%%%%%%%%%%%%%%%%%%%%%%%%%%%%%%%%%%%%%%%%%%%%%%%%%%%%%%%
\providecommand{\bysame}{\leavevmode\hbox to3em{\hrulefill}\thinspace}
\providecommand{\MR}[1]{}
\renewcommand{\MR}[1]{%
 MR \href{http://www.ams.org/mathscinet-getitem?mr=#1}{#1}.
}
\providecommand{\href}[2]{#2}
\newcommand{\articlehref}[2]{\href{#1}{#2}}

\end{document}